\documentclass[a4paper,12pt]{article}
\usepackage{amsmath,amssymb,amsthm,latexsym}
\theoremstyle{plain}
\newtheorem{theorem}{Theorem}

\title{The niche graphs of interval orders}

\author{
{\sc Jeongmi PARK}
\footnote{
Department of Mathematics, 
Pusan National University, 
Busan 609-735, Korea. 
}
\footnote{
{\it E-mail address}: {\tt jm1015@pusan.ac.kr}}
\and 
{\sc Yoshio SANO}
\footnote{
Division of Information Engineering, 
Faculty of Engineering, Information and Systems,  
University of Tsukuba, 
Ibaraki 305-8573, Japan.
}
\footnote{
{\it E-mail address}: {\tt sano@cs.tsukuba.ac.jp}} 
}

\date{}

\begin{document}

\maketitle

\begin{abstract}
The \emph{niche graph} 
of a digraph $D$ is the (simple undirected) graph 
which has the same vertex set as $D$ 
and has an edge between two distinct vertices $x$ and $y$ 
if and only if
$N^+_D(x) \cap N^+_D(y) \neq \emptyset$ 
or 
$N^-_D(x) \cap N^-_D(y) \neq \emptyset$, 
where $N^+_D(x)$ (resp. $N^-_D(x)$) 
is the set of out-neighbors (resp. in-neighbors) of $x$ in $D$. 
A digraph $D=(V,A)$ is called a 
\emph{semiorder} (or a \emph{unit interval order}) 
if there exist a real-valued function 
$f:V \to \mathbb{R}$ on the set $V$ and a positive real number 
$\delta \in \mathbb{R}$ 
such that $(x,y) \in A$ if and only if $f(x) > f(y) + \delta$. 
A digraph $D=(V,A)$ is called an 
\emph{interval order} if there exists 
an assignment $J$ 
of a closed real interval $J(x) \subset \mathbb{R}$
to each vertex $x \in V$ 
such that $(x,y) \in A$ if and only if $\min J(x) > \max J(y)$. 

S. -R. Kim and F. S. Roberts characterized the competition graphs of 
semiorders and interval orders in 2002, 
and 
Y. Sano characterized the competition-common enemy graphs of 
semiorders and interval orders in 2010. 
In this note, we give characterizations of the niche graphs 
of semiorders and interval orders. 
\end{abstract}




\section{Introduction}

J. E. Cohen \cite{Cohen} introduced the notion of 
competition graphs in 1968 in connection with a problem in ecology. 
The \emph{competition graph} $C(D)$ of a digraph $D$ 
is the (simple undirected) graph which 
has the same vertex set as $D$ 
and has an edge between two distinct vertices 
$x$ and $y$ if and only if $N^+_D(x) \cap N^+_D(y) \neq \emptyset$, 
where $N^+_D(x) = \{v \in V(D) \mid (x,v) \in A(D) \}$ 
is the set of out-neighbors of $x$ in $D$. 
(For a digraph $D$, we denote the vertex set and the arc set of $D$ 
by $V(D)$ and $A(D)$, respectively.) 
It has been one of the important research problems in the study of competition 
graphs to characterize the competition graphs of digraphs satisfying 
some specified conditions. 

A digraph $D=(V,A)$ is called a 
\emph{semiorder} (or a \emph{unit interval order}) 
if there exist a real-valued function 
$f:V \to \mathbb{R}$ on the set $V$ and a positive real number 
$\delta \in \mathbb{R}$ 
such that $(x,y) \in A$ if and only if $f(x) > f(y) + \delta$. 
A digraph $D=(V,A)$ is called an 
\emph{interval order} if there exists 
an assignment $J$ 
of a closed real interval $J(x) \subset \mathbb{R}$
to each vertex $x \in V$ 
such that $(x,y) \in A$ if and only if $\min J(x) > \max J(y)$. 
We call $J$ an \emph{interval assignment} of $D$. 
(See \cite{IO} for details on interval orders.) 

A \emph{complete graph} is a graph which has an edge 
between every pair of vertices. 
We denote the complete graph with $n$ vertices by $K_n$. 
An \emph{edgeless graph} is a graph which has no edges. 
We denote the edgeless graph with $n$ vertices by $I_n$. 
The \emph{(disjoint) union} of two graphs $G$ and $H$ 
is the graph $G \cup H$ 
whose vertex set is the disjoint union of the vertex sets of $G$ and $H$ 
and whose edge set is the disjoint union of the edge sets of $G$ and $H$. 

S. -R. Kim and F. S. Roberts characterized the competition graphs of 
semiorders 
and interval orders as follows: 

\begin{theorem}[\cite{KimRoberts}]\label{thm:KR}
Let $G$ be a graph. Then the following are equivalent. 
\begin{itemize}
\item[{\rm (a)}]
$G$ is the competition graph of a semiorder, 
\item[{\rm (b)}]
$G$ is the competition graph of an interval order, 
\item[{\rm (c)}]
$G = K_r \cup I_q$ where if $r \geq 2$ then $q \geq 1$. 
\qed
\end{itemize}
\end{theorem}

D. D. Scott \cite{Scott} 
introduced the \emph{competition-common enemy graphs} of digraphs in 1987 
as a variant of competition graphs. 
The \emph{competition-common enemy graph} 
of a digraph $D$ is the graph 
which has the same vertex set as $D$ 
and has an edge between two distinct vertices $x$ and $y$ 
if and only if both 
$N^+_D(x) \cap N^+_D(y) \neq \emptyset$ 
and 
$N^-_D(x) \cap N^-_D(y) \neq \emptyset$ hold, 
where 
$N^-_D(x) = \{v \in V(D) \mid (v,x) \in A(D) \}$ 
is the set of in-neighbors of $x$ in $D$. 

Y. Sano characterized the competition-common enemy graphs of semiorders 
and interval orders as follows: 

\begin{theorem}[\cite{Sano}]
Let $G$ be a graph. Then the following are equivalent. 
\begin{itemize}
\item[{\rm (a)}]
$G$ is the competition-common enemy graph 
of a semiorder, 
\item[{\rm (b)}]
$G$ is the competition-common enemy graph 
of an interval order, 
\item[{\rm (c)}]
$G = K_r \cup I_q$ where if $r \geq 2$ then $q \geq 2$. 
\qed
\end{itemize}
\end{theorem}

Niche graphs are another variant of competition graphs, 
which were introduced by 
C. Cable, K. F. Jones, J.R. Lundgren, and S. Seager \cite{CJLS}. 
The \emph{niche graph} of a digraph $D$ 
is the graph which has the same vertex set as $D$ 
and has an edge between two distinct vertices $x$ and $y$ 
if and only if 
$N^+_D(x) \cap N^+_D(y) \neq \emptyset$ 
or  
$N^-_D(x) \cap N^-_D(y) \neq \emptyset$.

In this note, we characterize 
the niche graphs of semiorders 
and interval orders. 
As a consequence, 
it turns out that the class of the niche graphs of interval orders 
is larger than the class of the niche graphs of semiorders. 
In fact, the graph $P_3 \cup I_1$ 
(the union of a path with three vertices and an isolated vertex) is 
the niche graph of an interval order, 
but $P_3 \cup I_1$ is not the niche graph of a semiorder.

\section{Main Results}

To state our main results, 
we first recall basic terminology in graph theory. 
The vertex set and the edge set of a graph $G$ 
are denoted by $V(G)$ and $E(G)$, respectively. 
The \emph{complement} of a graph $G$ 
is the graph $\overline{G}$ defined by 
$V(\overline{G}) = V(G)$ and 
$E(\overline{G}) = 
\{v v' \mid v, v' \in V(G), v \neq v', 
v v' \not\in E(G)\}$. 
For positive integers $m$ and $n$, 
the \emph{complete bipartite graph} $K_{m,n}$ 
is the graph defined by 
$V(K_{m,n})=X \cup Y$, where $|X|=m$ and $|Y|=n$, 
and 
$E(K_{m,n})=\{xy \mid x \in X, y \in Y \}$. 
We can observe that $\overline{K_{m,n}}=K_m \cup K_n$. 

The niche graphs of semiorders are characterized as follows. 

\begin{theorem}\label{thm:main-semi}
A graph $G$ is the niche graph of a semiorder 
if and only if 
$G$ is one of the following graphs: 
\begin{itemize}
\item[{\rm (i)}] 
an edgeless graph $I_q$; 
\item[{\rm (ii)}] 
the union of two complete graphs $K_{m} \cup K_{n}$; 
\item[{\rm (iii)}] 
the union of two complete graphs and an edgeless graph 
$K_{m} \cup K_{n} \cup I_q$; 
\item[{\rm (iv)}] 
the complement of the union of a complete bipartite graph and
an edgeless graph $\overline{K_{m,n} \cup I_q}$, 
\end{itemize}
where $m$, $n$, and $q$ are positive integers. 
\end{theorem}

\begin{proof}
First, we show the ``only if" part. 
Let $G$ be the niche graph of a semiorder $D$. 
Then there exist a function $f:V(D) \to {\mathbb{R}}$ 
and a positive real number $\delta \in \mathbb{R}_{>0}$ 
such that $A(D) = \{ (x,y) \mid x,y \in V(D), 
f(x) > f(y) + \delta \}$. 
Let 
$r_1$ and $r_2$ be real numbers defined by 
\[
r_1 = \min_{x \in V(D)} f(x) 
\qquad \textrm{and} \qquad 
r_2 = \max_{x \in V(D)} f(x). 
\]
We consider the following three cases: 
(Case 1) $r_1+ \delta \geq r_2$; 
(Case 2) $r_1+ \delta < r_2 \leq r_1+ 2\delta$; 
(Case 3) $r_1+ 2\delta < r_2$.

(Case 1) Consider the case where $r_1+ \delta \geq r_2$. 
In this case, we can observe that $D$ has no arcs. 
Therefore $G$ is an edgeless graph. 

(Case 2) Consider the case where $r_1+ \delta < r_2 \leq r_1+ 2\delta$. 
Note that $r_1 < r_2 - \delta \leq r_1 + \delta <r_2$. 
Let 
$V_1$, $V_2$, and $V_3$ be subsets of $V(D)$ defined by 
\begin{eqnarray*}
V_1 &=& \{ v \in V(D) \mid r_1 \leq f(v) < r_2 - \delta \}, \\
V_2 &=& \{ v \in V(D) \mid r_2 - \delta \leq f(v) \leq r_1 + \delta \}, \\
V_3 &=& \{ v \in V(D) \mid r_1 + \delta < f(v) \leq r_2 \}. 
\end{eqnarray*}
Then it follows that $V(G)=V_1 \cup V_2 \cup V_3$ 
and $V_i \cap V_j = \emptyset$ if $i \neq j$. 
Note that $V_1 \neq \emptyset$ 
since there exists a vertex $x \in V(D)$ 
such that $f(x)=r_1$, 
and that $V_3 \neq \emptyset$ 
since there exists a vertex $x \in V(D)$ 
such that $f(x)=r_2$. 
The set $V_1$ forms a clique in $G$ 
since any vertex in $V_1$ has 
a common in-neighbor which belongs to $V_3$ in $D$. 
The set $V_3$ forms a clique in $G$ 
since any vertex in $V_3$ has 
a common out-neighbor which belongs to $V_1$ in $D$. 
Any vertex in $V_1$ and any vertex in $V_3$ 
are not adjacent in $G$ since 
any vertex in $V_1$ has no out-neighbor in $D$ 
and 
any vertex in $V_3$ has no in-neighbor in $D$. 
Furthermore, any vertex in the set $V_2$ is an isolated vertex in $G$ 
since it has neither an in-neighbor nor an out-neighbor in $D$. 
That is, the set $V_2$ induces an edgeless graph if $V_2 \neq \emptyset$. 
Thus, $G$ is the union of two complete graphs, or
$G$ is the union of two complete graphs and an edgeless graph. 

(Case 3) 
Consider the case where $r_1+ 2\delta < r_2$. 
Note that $r_1 < r_1 + \delta < r_2 - \delta < r_2$. 
Let 
$V_1$, $V_2$, and $V_3$ be subsets of $V(D)$ defined by 
\begin{eqnarray*}
V_1 &=& \{ v \in V(D) \mid r_1 \leq f(v) \leq r_1 + \delta \}, \\
V_2 &=& \{ v \in V(D) \mid r_1 + \delta < f(v) < r_2 - \delta \}, \\
V_3 &=& \{ v \in V(D) \mid r_2 - \delta \leq f(v) \leq r_2 \}. 
\end{eqnarray*}
Then it follows that $V(G)=V_1 \cup V_2 \cup V_3$ 
and $V_i \cap V_j = \emptyset$ if $i \neq j$. 
Note that $V_1 \neq \emptyset$ 
and $V_3 \neq \emptyset$. 
The set $V_2 \cup V_3$ forms a clique in $G$ 
since any vertex in $V_2 \cup V_3$ has 
a common out-neighbor which belongs to $V_1$ in $D$. 
The set $V_1 \cup V_2$ forms a clique in $G$ 
since any vertex in $V_1 \cup V_2$ has 
a common in-neighbor which belongs to $V_3$ in $D$. 
Any vertex in $V_1$ and any vertex in $V_3$ 
are not adjacent in $G$ since 
any vertex in $V_1$ has no out-neighbor in $D$ 
and 
any vertex in $V_3$ has no in-neighbor in $D$. 
Therefore, $G=\overline{K_{m,n} \cup I_q}$ 
where $m=|V_1|$, $n=|V_3|$, and $q=|V_2|$.  
Thus, $G$ is the union of two complete graphs 
if $V_2 = \emptyset$, 
and $G$ is 
the complement of the union of a complete bipartite graph and
an edgeless graph if $V_2 \neq \emptyset$. 

Second, we show the ``if" part. 
Case (i): 
Let $G$ be an edgeless graph. 
We define a function $f:V(G) \to \mathbb{R}$ by
$f(x) = 1$ for all $x \in V(G)$, and let $\delta=1$. 
Then $f$ and $\delta$ gives a semiorder $D=(V,A)$ 
where $V=V(G)$ and $A= \emptyset$, 
and the niche graph of the semiorder $D$ is 
the graph $G$. 
Cases (ii) and (iii): 
Let $G$ be the union of two complete graphs $K$ and $K'$ 
and an edgeless graph $I$, 
where $I$ may possibly be the graph with no vertices. 
We define a function $f:V(G) \to \mathbb{R}$ by 
$f(x) = 1$ if $x \in V(K)$, 
$f(x) = 4$ if $x \in V(K')$, 
$f(x) = 2$ if $x \in V(I)$, 
and let $\delta=2$. 
Then $f$ and $\delta$ gives a semiorder $D=(V,A)$ 
where $V = V(G)$ 
and $A=\{(x, y) \mid x \in V(K'), y \in V(K) \}$, 
and the niche graph of the semiorder $D$ is 
the graph $G$. 
Case 
(iv): 
Let $G = \overline{K_{m,n} \cup I_q}$. 
Let $X$ and $Y$ be the partite sets of the complete bipartite graph $K_{m,n}$ 
and let $Z$ be the vertex set of the edgeless graph $I_q$. 
Then $(X,Y,Z)$ is a tripartition of the vertex set of $G$ 
and 
$E(G)= \{v v' \mid v, v' \in V(G), v \neq v' \}
\setminus 
\{ x y \mid x \in X, y \in Y \}$. 
Now, we define a function $f:V(G) \to \mathbb{R}$ by 
$f(x) = 1$ if $x \in X$, 
$f(z) = 3$ if $z \in Z$, 
$f(y) = 5$ if $y \in Y$, 
and let $\delta=1$. 
Then $f$ and $\delta$ gives a semiorder $D=(V,A)$ 
where $V = V(G)$ 
and 
$A=\{(y, x) \mid x \in X, y \in Y \} 
\cup \{(z, x) \mid x \in X, z \in Z \}
\cup \{(y, z) \mid z \in Z, y \in Y \}$, 
and the niche graph of the semiorder $D$ is 
the graph $G$. 
Hence the theorem holds. 
\end{proof}

The next theorem characterizes 
the niche graphs of interval orders. 

\begin{theorem}\label{thm:main-int}
A graph $G$ is the niche graph of an interval order 
if and only if 
$G$ is one of the following graphs: 
\begin{itemize}
\item[{\rm (i)}] 
an edgeless graph $I_q$; 
\item[{\rm (ii)}] 
the union of two complete graphs $K_{m} \cup K_{n}$; 
\item[{\rm (iii)}] 
the union of two complete graphs and an edgeless graph 
$K_{m} \cup K_{n} \cup I_r$; 
\item[{\rm (iv)}] 
the complement of the union of a complete bipartite graph and
an edgeless graph $\overline{K_{m,n} \cup I_q}$; 
\item[{\rm (v)}] 
the union of an edgeless graph and 
the complement of the union of a complete bipartite graph and 
an edge less graph $I_{r} \cup \overline{K_{m,n} \cup I_{q}}$, 
\end{itemize}
where $m$, $n$, $q$, and $r$ are positive integers. 
\end{theorem}

\begin{proof}
For positive integers $m$ and $n$ and 
non-negative integers $q$ and $r$, 
let 
\[
\Gamma(m,n,q,r) = \overline{K_{m,n} \cup I_{q}} \cup I_{r}. 
\]
We remark that 
$\Gamma(m,n,0,0)= K_{m} \cup K_{n}$, 
$\Gamma(m,n,0,r)= K_{m} \cup K_{n} \cup I_r$, and 
$\Gamma(m,n,q,0)= \overline{K_{m,n} \cup I_q}$. 

First, we show the ``only if" part. 
Let $G$ be the niche graph of an interval order $D$. 
Then there exists an interval assignment 
$J$ 
of $D$. 
Let 
$r_1$ and $r_2$ be real numbers defined by 
\[
r_1 = \min_{x \in V(D)} \max J(x) 
\qquad \textrm{and} \qquad 
r_2 = \max_{x \in V(D)} \min J(x). 
\]
If $r_1 \geq r_2$, then 
we can observe that $D$ has no arcs 
and therefore $G$ is an edgeless graph. 
Now, we consider the case where $r_1 < r_2$. 
Note that $|V(G)| \geq 2$ since $r_1$ and $r_2$ 
are attained by different vertices. 
Let 
$V_1$, $V_2$, $V_3$, and $V_4$ be subsets of $V(D)$ 
defined by 
\begin{eqnarray*}
V_1 &=& \{ v \in V(D) \mid \min J(v) \leq r_1 \leq \max J(v) < r_2 \}, \\
V_2 &=& \{ v \in V(D) \mid r_1 < \min J(v), \max J(v) < r_2 \}, \\
V_3 &=& \{ v \in V(D) \mid r_1 < \min J(v) \leq r_2 \leq \max J(v) \}, \\
V_4 &=& \{ v \in V(D) \mid \min J(v) \leq r_1, r_2 \leq \max J(v) \}. 
\end{eqnarray*}
Then it follows that $V(G)=V_1 \cup V_2 \cup V_3 \cup V_4$ 
and $V_i \cap V_j = \emptyset$ if $i \neq j$. 
Note that $V_1 \neq \emptyset$ 
since there exists a vertex $x \in V(D)$
such that $\max J(x) = r_1$, 
and that 
$V_3 \neq \emptyset$ 
since there exists a vertex $x \in V(D)$
such that $\min J(x) = r_2$. 
The set $V_2 \cup V_3$ forms a clique in $G$ 
since any vertex in $V_2 \cup V_3$ has 
a common out-neighbor which belongs to $V_1$ in $D$. 
The set $V_1 \cup V_2$ forms a clique in $G$ 
since any vertex in $V_1 \cup V_2$ has 
a common in-neighbor which belongs to $V_3$ in $D$. 
Any vertex in $V_1$ and any vertex in $V_3$ 
are not adjacent in $G$ since 
any vertex in $V_1$ has no out-neighbor in $D$ 
and 
any vertex in $V_3$ has no in-neighbor in $D$. 
Therefore, the set $V_1 \cup V_2 \cup V_3$ 
induces the graph 
$\overline{K_{m,n} \cup I_q}$ 
where $m=|V_1|$,  $n=|V_3|$, and $q=|V_2|$. 
Furthermore, any vertex in the set $V_4$ is an isolated vertex in $G$ 
since it has neither an in-neighbor nor an out-neighbor in $D$. 
That is, the set $V_4$ induces an edgeless graph. 
Thus $G$ is 
the graph $\Gamma(m,n,q,r)$ with 
$m=|V_1|$, $n=|V_3|$, $q=|V_2|$, and $r=|V_4|$. 

Second, we show the ``if" part. 
Case 
(i): 
Let $G$ be an edgeless graph. 
We define an interval assignment 
$J$ 
by $J(x) = [1,2]$ for all $x \in V(G)$, 
where $[a,b]$ denotes the closed real interval 
$\{r \in \mathbb{R} \mid a \leq r \leq b \}$. 
Then $J$ gives an interval order $D=(V,A)$ 
where $V=V(G)$ and $A= \emptyset$, 
and the niche graph of the semiorder $D$ is 
the graph $G$. 
Cases 
(ii)-(v): 
Let $G$ be the graph $\Gamma(m,n,q,r)$ 
for some positive integers $m$ and $n$ 
and non-negative integers $q$ and $r$. 
Then, there exists 
a partition $(U_1,U_2,U_3,U_4)$ of the vertex set of $G$ 
such that 
$E(G)= \{v v' \mid v, v' \in U_1 \cup U_2 \cup U_3, v \neq v' \}
\setminus 
\{ u_1 u_3 \mid u_1 \in U_1, u_3 \in U_3 \}$. 
Note that $\{|U_1|,|U_3|\}=\{m,n\}$, $|U_2|=q$, and $|U_4|=r$. 
Now, we define an interval assignment 
$J$  
as follows: 
$J(x) = [1,2]$ if $x \in U_1$; 
$J(x) = [3,4]$ if $x \in U_2$; 
$J(x) = [5,6]$ if $x \in U_3$; 
$J(x) = [1,6]$ if $x \in U_4$. 
Then $J$ gives an interval order $D=(V,A)$ 
where $V = V(G)$ 
and $A=\{(x, y) \mid x \in U_i, y \in U_j, 
(i,j) \in \{(3,2), (3,1), (2,1) \} \}$, 
and the niche graph of the interval order $D$ is 
the graph $G$. 
Hence the theorem holds. 
\end{proof}


\end{document}